\documentclass[a4paper,12pt]{amsart}

\usepackage{amsfonts,amsthm,amssymb,amsmath}
\usepackage{graphicx}
\usepackage{color}
\usepackage[all,cmtip]{xy}

\usepackage{t1enc}
\usepackage[latin1]{inputenc}

\newtheorem{theorem}{Theorem}[section]
\newtheorem{proposition}[theorem]{Proposition}
\newtheorem{corollary}[theorem]{Corollary}
\newtheorem{lemma}[theorem]{Lemma}
\newtheorem{definition}[theorem]{Definition}
\newtheorem{remark}[theorem]{Remark}

\newcommand{\N}{\mathbb{N}}
\newcommand{\Nuc}{\mathcal{N}}
\newcommand{\U}{\mathfrak A}
\newcommand{\B}{\mathcal{L}}
\newcommand{\E}{\mathcal{E}}
\newcommand{\I}{\mathcal{I}}
\newcommand{\ep}{\ell_{p}}
\newcommand{\eq}{\ell_{q}}

\DeclareMathOperator{\id}{id}

\begin{document}

\title{Diagonal extendible multilinear operators between $\ell_{p}$-spaces}

\author{Daniel Carando
\and
Ver\'{o}nica Dimant
\and
Pablo Sevilla-Peris
\and
Rom\'an Villafa\~ne}

\thanks{The first, second and fourth authors were partially supported by CONICET PIP 0624 and UBACyT 20020100100746.
The third author was supported by MICINN Project MTM2011-22417. The fourth author has a doctoral fellowship from CONICET}

\date{}


\address{Departamento de Matem\'{a}tica - Pab I,
Facultad de Cs. Exactas y Naturales, Universidad de Buenos Aires,
(C1428EGA) Buenos Aires, Argentina and IMAS -  CONICET.} \email{dcarando@dm.uba.ar}

\address{Departamento de Matem\'{a}tica, Universidad de San
Andr\'{e}s, Vito Dumas 284, (B1644BID) Victoria, Buenos Aires,
Argentina and CONICET.} \email{vero@udesa.edu.ar}

\address{Instituto Universitario de Matem\'{a}tica Pura y Aplicada and DMA, ETSIAMN, Universitat Polit\`{e}cnica de Val\`{e}ncia, Valencia, Spain} \email{psevilla@mat.upv.es}

\address{Departamento de Matem\'{a}tica - Pab I,
Facultad de Cs. Exactas y Naturales, Universidad de Buenos Aires,
(C1428EGA) Buenos Aires, Argentina and IMAS - CONICET.} \email{rvillafa@dm.uba.ar}

\begin{abstract}We study extendibility of diagonal multilinear operators from $\ell_p$ to $\ell_q$ spaces. We determine the values of $p$ and $q$
for which every diagonal $n$-linear operator is extendible, and those for which the only extendible ones are integral. We address the same question
for multilinear forms on $\ell_p$.
\end{abstract}

\maketitle

\section*{Introduction}

There is no Hahn-Banach Theorem for linear/multilinear operators, nor for multilinear forms. This makes \textit{extendibility} of such mappings a subject of
interest, which has connections with different branches of functional analysis and Banach space theory (local theory, tensor norms, geometry of Banach spaces,
etc.). A multilinear operator (on some Banach space) $T \in \mathcal{L} (^{n} E,F)$ is \textit{extendible} (see \cite{Ca01,CaGaJa01,JaPaPGVi07,KiRy98}) if for every Banach space $X$ containing
$E$ there exists $\tilde{T} \in  \mathcal{L} (^{n} X,F)$ extending $T$. For fixed Banach spaces $E$ and $F$, one may wonder which multilinear operators from $E$ to $F$ are extendible. There are two extreme cases which are particularly important: sometimes extendible $n$-linear operators are ``as few as possible'' and sometimes they are ``as many as possible''. More precisely, since integral multilinear operators are always extendible, the ``as few as  possible''-case occurs when only the integral mappings are extendible. On the other hand, the ``as many as possible''-situation is when all the
multilinear operators are extendible. This problem, for bilinear forms on Banach sequence spaces, has been recently addressed in~\cite{CaSe2}.

In this article we focus on particular sequence spaces: $\ell_p$-spaces.
The study of diagonal operators on $\ell_p$-spaces started in the early 1970's with the work of (among others) Carl \cite{Carl76}, K\"onig \cite{Ko75} and Pietsch \cite{Pi78} and by
now they are a well established part of the theory. On the other hand, in \cite{Ca01, CaDiSe} we have studied diagonal multilinear forms on these spaces
(see also \cite{CaDiSe3} where multilinear forms on Lorentz sequence spaces were also considered).
In this paper we carry on with the study of diagonal multilinear forms, and also extend
it to diagonal multilinear operators.
In this context, we want to characterize the extendible elements, and the extreme cases mentioned above must be rephrased for diagonal mappings.
Thus, we want to determine which values of $p$, $q$ and $n$ make the set of diagonal extendible $n$-linear operators from $\ell_p$ to $\ell_q$ equal  to the set of integral mappings, which make it  equal to {{the whole}} space of continuous diagonal $n$-linear operators from $\ell_p$ to $\ell_q$, and which make it something in between. For this, we relate properties of the multilinear mapping to summability properties of its coefficients. The conclusions of the main result of the article (Theorem~\ref{main}) are illustrated in Tables \ref{tablanucint} and \ref{tablaexten} in Section 1.

Section 2 is devoted to diagonal multilinear forms. Several steps in the scalar valued settings were given in \cite{Ca01, CaDiSe}.
It is known that in $\ell_{1}$ every continuous  (and, in particular, every extendible) diagonal form is integral.
Also,  every diagonal form  on $c_0$ or $\ell_{\infty}$ is nuclear. For the non-trivial range of $p$,  Proposition~3.1 in \cite{CaDiSe} shows that for $p \geq 2$ a diagonal
$n$-linear form on $\ell_{p}$ is extendible if and only if it is nuclear. Corollary~3.1 in \cite{CaDiSe} shows that  there exists
a diagonal $n$-linear form that is extendible but not nuclear in every $\ell_{p}$ with $\frac{2(n-1)}{2n-3} < p < 2$,
leaving unsolved the question for the remaining values of $p$. We complete here these results. For example, a consequence of Theorem~\ref{main forms} is that there are extendible diagonal $n$-linear forms that are not integral on  $\ell_p$ for every $1<p<2$ and $n\ge 3$. Another one is that there are non-extendible diagonal $n$-linear forms in every $\ell_p$ for $1<p<\infty$.

Our main theorem, regarding multilinear operators, and its proof is the {content} of Section~3. This result
gives a rather complete account of the summability conditions of the coefficients of a multilinear operator $T$ from $\ell_p$ to $\ell_q$  which are necessary and sufficient for $T$ being integral or extendible. As a consequence, the existence (or lack) of diagonal multilinear operators which are not extendible, or which are extendible but not integral, is established for every $p$ and $q$ (see Table~\ref{tablaexten}).

As a byproduct of some of our results we show that spaces of diagonal multilinear forms/operators behave very differently to spaces of all multilinear forms/operators. For example, there are always non-extendible diagonal bilinear forms on $\ell_p\times \ell_1$ if  $p>1$ but, surprisingly,
 every diagonal trilinear form on $\ell_p\times \ell_1\times\ell_1 $ is extendible (see the comments after Lemma~\ref{ejemplo todas extendibles}). See Proposition~\ref{integrales nucleares} and the subsequent comments for another unexpected behaviour.

For the theory of polynomials and multilinear mappings on Banach spaces we refer the reader to the books of Dineen \cite{Di99} and Mujica \cite{Mu86}.

\section{Preliminaries}
We consider Banach spaces over $\mathbb K=\mathbb R$ or $\mathbb C$. For simplicity, we write some of our proofs for complex spaces, but all the results hold for both the real and complex  cases.
For Banach spaces $E$ and $F$, we denote by $\mathcal L(^nE,F)$ the space  of continuous $n$-linear mappings from $E\times\cdots\times E$  into $F$. This is a Banach space if we consider the norm
$$
\|T\|=\sup\{\|T(x_1,\dots, x_n)\|:\, x_i\in B_E,\, i=1,\dots,  n\}.
$$
Here $B_{E}$ denotes the closed unit ball of $E$.\\

We are going to work mainly with Minkowski $\ell_{p}$-spaces and  nuclear, integral and extendible multilinear mappings between them. We recall their definitions. Given $1 < p < \infty$, we denote its conjugate by $p'$, that is the number satisfying $1= \frac{1}{p} + \frac{1}{p'}$. As usual, $1$ and $\infty$ are conjugate to each other.\\

An $n$-linear mapping $T\in\mathcal L(^n \ep, \eq)$  is said to be \textbf{nuclear} (see \cite{Al85})  if it can be written as
\begin{equation}\label{def:nuclear}
T(x_1,\dots, x_n)=\sum_{j=1}^\infty \gamma_1^{(j)}(x_1)\cdots \gamma_n^{(j)}(x_n)y_j \text{ for all } x_1,\dots, x_n\in \ep,
\end{equation}
where $\gamma_i^{(j)}\in \ell_{p'}$, $y_j\in \eq$, for every $i,j$, and $\sum_{j=1}^\infty \|\gamma_1^{(j)}\|\cdots \|\gamma_n^{(j)}\|\cdot \|y_j\|<\infty$. We denote by
$\mathcal{N}(^n \ep,\eq)$  the space of nuclear $n$-linear mappings from $\ep\times\cdots\times \ep$  to $\eq$. This is a Banach space if we endow it with the norm
$$
\|T\|_{_{\mathcal N}}=\inf\left\{\sum_{j=1}^\infty \|\gamma_1^{(j)}\|\cdots \|\gamma_n^{(j)}\|\cdot \|y_j\|\right\},
$$
where  the infimum is taken over all the nuclear representations of $T$ as in (\ref{def:nuclear}).

An $n$-linear mapping $T\in\mathcal L(^n \ep, \eq)$  is said to be \textbf{integral} if there exists a regular $\eq$-valued measure $G$ of bounded
variation on the product $B_{\ell_{p'}}\times\cdots\times B_{\ell_{p'}}$ such that
$$
T(x_1,\dots, x_n)=\int_{B_{\ell_{p'}}\times\cdots\times B_{\ell_{p'}}}\gamma_1(x_1)\cdots \gamma_n(x_n)\, dG(\gamma_1,\dots,\gamma_n),
$$
for all $(x_1,\dots, x_n)\in \ep \times \cdots \times \ep$. The integral norm  of $T$ is defined as the infimum of the total variation of $G$ over all measures $G$ representing $T$. With this norm $\mathcal{I}(^n \ep, \eq)$,  the space of all integral $n$-linear mappings from $\ep \times\cdots\times \ep$  to $\eq$
is a Banach space.\\
There are two general definitions of integral $n$-linear operators in the literature (Grothendieck and Pietsch integral) but it is known \cite{Vi} that in our context both coincide.

We denote by $\mathcal{E}(^n \ep, \eq)$ the space consisting of all extendible mappings in $\mathcal{L}(^n \ep, \eq)$; it is a Banach space if we endow it with the norm
\begin{multline*}
\|T\|_{_{\mathcal E}}=\inf\{C>0:\\ \forall E \supset \ep \text{ there exists an extension }\widetilde{T}\in\mathcal L(^n E,\eq)\text{ with } \|\widetilde T\|\le C\}.
\end{multline*}

Every nuclear $n$-linear mapping is clearly integral. The reverse inclusion sometimes holds: for $1<p<\infty$, the space $\ell_p$ is Asplund and then integral multilinear operators on $\ell_p$ are nuclear \cite{Al85}. On the other hand, by \cite{CaLa}, integral $n$-linear operators are
extendible. Thus, we  have the following chain of embeddings:
\begin{equation}\label{chain}
\mathcal N(^n \ep, \eq)\subseteq \mathcal I(^n \ep, \eq)\subseteq \mathcal E(^n \ep, \eq)\subseteq \mathcal L(^n \ep, \eq).
\end{equation}
Regarding the extendible operators we have two extreme cases: when only the integral operators are extendible or when all operators are so. Our main aim in this paper
is to try to characterize these two cases {for a class of  distinguished multilinear operators}, the so called `diagonal mappings'. We will also determine when nuclear
and integral diagonal mappings coincide (this last question has interest only for $p=1$ and $p=\infty$, by the above mentioned result in \cite{Al85}).\\

An $n$-linear operator $T \in\mathcal{L}(^{n} \ep, \eq)$ is said to be \textbf{diagonal} if there exists a sequence $\alpha=(\alpha(k))_{k}$ such that for all
$x_{1}, \dots,x_{n} \in \ep$ we can write
\begin{equation*}
T(x_{1}, \dots,x_{n}) = \sum_{k} \alpha(k) x_{1}(k) \cdots x_{n}(k)\ e_k \, ,
\end{equation*}
where $e_k$ denotes the $k$-th canonical unit vector: $e_k(j)=\delta_{k,j}$ for $k,j \in \mathbb{N}$. We denote by $T_\alpha$ the diagonal multilinear mapping given by the sequence $\alpha$.

Recall that a Banach sequence space is a Banach space $E \subseteq \mathbb{K}^{\mathbb{N}}$ of sequences in $\mathbb{K}$ such that
$\ell_{1} \subseteq E \subseteq \ell_{\infty}$ satisfying that if $x \in \mathbb{K}^{\mathbb{N}}$ and $y \in E$ are such that $\vert x(k) \vert \leq \vert y(k) \vert$
for all $k \in \mathbb{N}$ then $x \in E$ and $\Vert x \Vert \leq \Vert y \Vert$.
\begin{definition} \label{def elene}
For a Banach ideal of $n$-linear mappings $\mathfrak{A}=\mathcal{N,I,E,L}$ we define the set
$$
\ell_n(\U,p,q)=\{ \alpha\in\ell_\infty : T_\alpha\in\U(^n \ep , \eq) \},
$$
which is a Banach sequence space with the norm   $\|\alpha\|_{_{\ell_n(\U,p,q)}}=\|T_\alpha\|_{{\U(^n \ep, \eq)}}$.
\end{definition}
Then, $\ell_n(\U,p,q)$ describes the space of diagonal $n$-linear mappings from $\ep$ to $\eq$ that belong to the ideal $\mathfrak A$. For example, for  $\mathcal{L}$  the ideal of continuous multilinear
operators, it is easy to check using H\"older's inequality that
\begin{equation*} 
\ell_n(\mathcal L,p,q)\overset{1}{=} \left\{\begin{array}{cl}
  \ell_\infty \quad & \text{ if }p\leq nq \\
  \  \\
  \ell_r \quad & \text{ if }p> nq\\
\end{array}\right.%
\end{equation*}
where $r$ is defined by $\frac 1 r= \frac1q-\frac np$. The notation $E\overset{1}{=}F$ means that $E$ and  $F$ are isometrically isomorphic.
Then the chain of embeddings \eqref{chain} implies that
\begin{equation*}
\ell_n(\mathcal N,p,q)\subseteq\ell_n(\mathcal I,p,q)\subseteq \ell_n(\mathcal E,p,q)\subseteq\ell_n(\mathcal L,p,q)
\end{equation*}
Our aim is to study when these inclusions are strict or not. A description of these sequence spaces is given in  our main result
(Theorem~\ref{main}). As an immediate consequence of it we obtain the following tables.

\bigskip

\begin{center}
\begin{table}[h]
\begin{tabular}{|c|ccc|}
 \hline
$\ell_{n} (\mathcal{N},p,q) \neq \ell_{n} (\mathcal{I},p,q)$ & $p=1$& and & $q=\infty$ \\ \hline
$\ell_{n} (\mathcal{N},p,q) = \ell_{n} (\mathcal{I},p,q)$ &  \multicolumn{3}{c|}{otherwise} \\ \hline
\end{tabular}

\medskip
\caption{}
\label{tablanucint}
\end{table}

\begin{table}[h]
\begin{tabular}{|c|ccc|}
\hline
$\ell_{n} (\mathcal{I},p,q)= \ell_{n} (\mathcal{E},p,q)= \ell_{n} (\mathcal{L},p,q)$ & $p=1$& and & $q=\infty$ \\
 & $p=\infty$& and & $q=1$ \\ \hline
$\ell_{n} (\mathcal{I},p,q)= \ell_{n} (\mathcal{E},p,q) \neq \ell_{n} (\mathcal{L},p,q)$ & $2 \leq p < \infty$& and &  $q=1$ \\ \hline
$\ell_{n} (\mathcal{I},p,q) \neq \ell_{n} (\mathcal{E},p,q) = \ell_{n} (\mathcal{L},p,q)$ & $p=1$& and  & $1 \leq q < \infty$ \\
 & $1<p<\infty$& and  & $q=\infty$ \\
 & $p=\infty$& and  & $1< q \leq \infty$ \\ \hline
 $\ell_{n} (\mathcal{I},p,q) \neq \ell_{n} (\mathcal{E},p,q) \neq \ell_{n} (\mathcal{L},p,q)$ & $1<p<2$& and & $q=1$ \\
  & $1< p < \infty$& and  & $1 < q < \infty$ \\ \hline
\end{tabular}
\medskip

\caption{}
\label{tablaexten}
\end{table}
\end{center}


\section{Diagonal multilinear forms in $\ell_p$ spaces}
There is a natural isometric identification between $\mathcal L(^n \ep, \eq)$ and the space of continuous $(n+1)$-linear applications from
$\ep \times \cdots \times \ep \times \ell_{q'}$ to $\mathbb K$ (for $q=1$, we put $c_0$ instead of $\ell_\infty$ for the identification to be onto). The definitions of nuclear, integral and extendible $n$-linear operator can be
modified in an obvious way to consider mappings defined on $\ell_{p_{1}} \times \cdots \times \ell_{p_{n}}$ with values on $\mathbb{K}$. It can be easily seen that this identification is also an isometric isomorphism for the classes of nuclear and integral mappings. However, this is not the case for the extendible operators. An extendible $(n+1)$-linear
form defined on $\ep \times \cdots \times \ep \times \ell_{q'}$ produces an extendible $n$-linear operator from $\ep \times \cdots \times \ep$ to $\eq$ (since $\ell_1$ is complemented in its bidual, for $q=1$ we can use either $\ell_\infty$ or $c_0$ instead of $\ell_{q'}$). However,  the
converse is not true, as we will see in Remark~\ref{identification}. Anyway, it will be helpful for our general goal to look first at multilinear forms.

One of the aims in \cite{Ca01,CaDiSe} was precisely to determine when the diagonal extendible multilinear forms are few (that is, they coincide with
the integral multilinear forms) and when they {are as many as they can} (all continuous multilinear forms are extendible). The case  $1<p<2$ was left open and we fill this
gap here.\\

For a  sequence $\alpha$, the  diagonal $n$-linear form  $\phi_{\alpha} \in\mathcal{L}(^{n} \ep)$ associated to $\alpha$ is given by
$$
\phi_{\alpha}(x_{1}, \dots,x_{n}) = \sum_{k} \alpha(k) x_{1}(k) \cdots x_{n}(k)
$$
for  $x_{1}, \dots,x_{n} \in \ep$ (whenever this mappings is well defined). Also, like in Definition~\ref{def elene}, for $\mathfrak{A}=\mathcal{N,I,E,L}$ we define
$\ell_n(\U,p)=\{ \alpha\in\ell_\infty : \phi_\alpha\in\U(^n \ep) \}$, which again is a Banach sequence space with the norm
$\|\alpha\|_{_{\ell_n(\U,p)}}=\|\phi_\alpha\|_{_{\U(^n\ell_p)}}$.

The following theorem describes the sequence spaces $\ell_n(\mathcal N,p)$,  $\ell_n(\mathcal I,p)$, $\ell_n(\mathcal E,p)$ and $\ell_n(\mathcal L,p)$.
One consequence is that the inclusion  $\ell_n (\mathcal I,p)\subset \ell_n(\mathcal E,p)$ is strict for all $1<p<2$ and $n\ge 3$, while the inclusion $\ell_n(\mathcal E,p)\subset \ell_n(\mathcal L,p)$ is strict for $1<p<\infty$ and $n\ge 2$.
\begin{theorem} \label{main forms}
 $\ $
\begin{enumerate}
\item $\ell_n(\mathcal N,1)=c_0\subsetneq\ell_\infty=\ell_n(\mathcal I,1)=\ell_n(\mathcal E,1)= \ell_n(\mathcal L,1) $.

\item If $1<p< 2$ then
\begin{gather*}
\ell_2(\mathcal N,p)=\ell_2(\mathcal I,p)=\ell_2(\mathcal E,p)=\ell_{\frac{p'}{2}} \subsetneq\ell_\infty = \ell_2(\mathcal L,p),  \\
\text{For }n\ge 3,\quad\ell_n(\mathcal N,p)=\ell_n(\mathcal I,p)=\ell_{\max (\frac{p'}{n} , 1)} {\subsetneq}\ell_{\frac{p'}{2}}=\ell_n(\mathcal E,p)\subsetneq\ell_\infty= \ell_n(\mathcal L,p).
\end{gather*}

\item If $2\le p<\infty$ then
$$
\ell_n(\mathcal N,p)=\ell_n(\mathcal I,p)=\ell_n(\mathcal E,p)=\ell_1 \subsetneq \ell_n(\mathcal L,p)= \left\{
\begin{array}{l}
 \ell_{\frac{p}{p-n}} \text{ for } n< p \,.\\
\ell_\infty \text{ for } n \geq p \, .
\end{array}
\right.
$$
\item $\ell_n(\mathcal N,\infty)=\ell_n(\mathcal I,\infty)=\ell_n(\mathcal E,\infty)= \ell_n(\mathcal L,\infty)=\ell_1$.
\end{enumerate}
\end{theorem}
The description of $\ell_{n} (\mathcal{L},p)$ follows easily from H\"older's inequality. Also, the characterizations of $\ell_{n}(\mathcal{N},p)$
and $\ell_{n} (\mathcal{I},p)$ and most of the relations with $\ell_{n} (\mathcal{E},p)$ where already proved in \cite{Ca01,CaDiSe}, but we have decided to bring
them all here to have a more complete picture. The only thing we need to prove is, in item \emph{(2)},  the equality $\ell_n(\mathcal E,p)=\ell_{\frac{p'}{2}}$.
This will be showed in Proposition \ref{pprimasobre2} below.
We remark that the case $1<p<2$ shows an important difference between bilinear forms and $n$-linear forms with $n\ge 3$. In fact, a consequence of a deep result of Pisier \cite{Pisier} is that every extendible bilinear form defined on a Banach space with cotype 2 is integral
(see  \cite{Ca01,CaGaJa01}, where Pisier's result is ``read'' in a fashion more akin to our framework). In particular, this holds for bilinear forms on $\ell_p$ with $1\le p\le 2$. However, by the previous theorem, for $n$-linear forms with $n\ge 3$ and every $1<p<2$ we  have diagonal extendible $n$-linear forms which are not integral.

For the following result, we need to define some mappings that will be used again later.
We consider a Walsh matrix; i.e., a matrix $(a_{kr})_{kr=1}^N$ such that $\vert a_{kr} \vert = 1$ and
\[
\sum_{r=1}^N a_{rk} \bar a_{rl}=N\ \delta_{k,l} \, .
\]
Assume also that the matrix is symmetric. Examples of such matrices in the complex case are the Fourier matrices, given by $a_{kr}=e^{\frac{2 \pi i}{N} rk}$ \cite[Section 8.5]{LibroDeFl93}. For the real
case one can consider Hadamard matrices (see \cite[Section 3]{MaQu10}) defined for $N$ that are powers of $2$ and that are generated by blocks as
follows
\[
 A_{2} = \begin{pmatrix}
          1 & 1 \\
	  1 & -1
         \end{pmatrix}
\, , \, \,
 A_{2^{n+1}} = \begin{pmatrix}
          A_{2^{n}} & A_{2^{n}} \\
	  A_{2^{n}} & -A_{2^{n}}
         \end{pmatrix}.
\]
We define a Toeplitz-like operator $\xi_N:\ell_p^N\to \ell_\infty ^N$ by
\begin{equation}\label{operador csi}
\xi_N(x)=\sum_{k=1}^N\left(\sum_{r=1}^N \bar a_{kr} x(r)\right)e_k.
\end{equation}
This operator satisfies $\Vert \xi_N:\ell_p^N\to\ell_\infty^N \Vert \leq \Vert \id : \ell_{p}^{N} \to \ell_{1}^{N} \Vert \,  \Vert \xi_{N} : \ell_{1}^{N} \to \ell_{\infty}^{N} \Vert \leq N^{\frac 1 {p'}}$.
Finally, for $n\ge 3$ we introduce the following modification of the $n$-linear form on $\ell_\infty^N$ studied by Bohnenblust and Hille in \cite[Section~2]{BoHi}:
\begin{equation}\label{trilineal L}
L_N(x_1,\dots,x_n)=\sum_{j,k,l=1}^{N}   a_{jl} a_{lk} x_1(j) x_2(k) x_3(l)\cdots x_n(l).
\end{equation}
Actually, this $n$-linear form is built taking the trilinear mapping defined by Bohnenblust and Hille and giving an $n$-linear extension of this definition maintaining
the norm of the original trilinear form. Since on $\ell_\infty^N$ the extendible and the usual norms coincide, we have by \cite[Section~2]{BoHi}
\begin{equation}\label{B-H}
\|L_N\|_{_{\mathcal E( ^n\ell_\infty^N)}}=\|L_N\|_{_{\mathcal L(^n \ell_\infty^N)}}=N^2 \, .
\end{equation}

\noindent Thanks to the properties of the coefficients $(a_{kr})_{kr}$  we have:
\begin{eqnarray}
L_N (\xi_N(x_1),\xi_N(x_2),x_3,\dots ,x_n  ) & = &  \sum_{j,k,l=1}^{N}   a_{jl} a_{lk} \cdot \xi_N(x_1)(j) \cdot \xi_N(x_2)(k)\cdot x_3(l)\cdots x_n(l)  \nonumber\\
&=&  \sum_{r,s,l=1}^{N}     x_1(r)  x_2(s) x_3(l)\cdots x_n(l)  \sum_{j=1}^{N}  a_{jl} \bar a_{jr} \sum_{k=1}^{N}   a_{lk} \bar a_{ks}\nonumber\\
&=& \sum_{r,s,l=1}^{N}     x_1(r)  x_2(s) x_3(l)\cdots x_n(l) N \delta_{l,r} N\delta_{l,s}\nonumber\\
&=& N^2 \sum_{r=1}^{N}     x_1(r)  x_2(r) x_3(r)\cdots x_n(r) \, .\label{L_N con csi_N}
\end{eqnarray}

\begin{lemma}\label{ejemplo todas extendibles}
Every diagonal $n$-linear form on $\ell_1\times \ell_1\times\ell_{p_1} \times\cdots\times \ell_{p_{n-2}} $ with $1\le p_i\le \infty$ and $n\ge 2$, is extendible.
\end{lemma}
\begin{proof} The case $n=2$ is immediate: diagonal bilinear forms on $\ell_1\times \ell_1$ are integral, and therefore extendible.
For $n\ge 3$, any diagonal $n$-linear form $\phi_\alpha:\ell_1\times\ell_1\times\ell_{p_1}\times\cdots\times\ell_{p_{n-2}}\to \mathbb K$ has a factorization as
$$
\xymatrix@C1.5mm{
\ell_1 \ar[d]_{D_\alpha} & \times & \ell_1 \ar[d]_i &\times & \ell_{p_1}\ar[d]_i &\times& \cdots & \times & \ell_{p_{n-2}}\ar[d]_i   \ar[rr]^{\phi_\alpha} & & \mathbb K \\
\ell_1&  \times &  \ell_1&\times& \ell_{p_1}&\times & \cdots & \times & \ell_{p_{n-2}}\ar[urr]_{\Phi}
},
$$
where  $D_\alpha:\ell_1\to\ell_1$ is the diagonal operator given by $D_{\alpha} ((x(k))_{k}) = (\alpha(k) x(k))_{k}$ and
\begin{equation} \label{def:fi}
\Phi(x_1,\dots,x_n)=\sum_{k=1}^\infty x_1(k)\cdots x_n(k).
\end{equation}
Then by the ideal property of extendible multilinear forms, it is enough to show that the $n$-linear form
$\Phi$ is extendible on $\ell_1\times \ell_1\times\ell_{p_1} \times\cdots\times \ell_{p_{n-2}} $.
We denote by $\Phi_{N}$ the form defined by summing the first $N$ terms in \eqref{def:fi}.

We consider the operator $\xi_N$ defined in (\ref{operador csi}) with domain $\ell_1$ (thus $\|\xi_N\|\le 1$) and the $n$-linear form $L_N$  as in (\ref{trilineal L}).
As a consequence of \eqref{L_N con csi_N} and \eqref{B-H}  we have
\begin{eqnarray*}
  \|\Phi_N\|_{_{\mathcal{E}(^n\ell_1\times \ell_1\times\ell_{p_1} \times\cdots\times \ell_{p_{n-2}})}} &\le &\frac 1{N^2 } \|L_N\|_{_{\mathcal E( ^n\ell_\infty^N)}} \|\xi_N\|^2 \ \prod_{i=1}^{n-2}\|\id :\ell_{p_i}^N\to\ell_\infty^N\|\ \\
 &\le & \frac 1{N^2 } N^2 = 1 \, .
\end{eqnarray*}
Therefore, the extendible norms of the $n$-linear forms $\Phi_N$ ($N\in \mathbb N$) are uniformly bounded. A multilinear  version of the density lemma \cite[Section~13.4]{LibroDeFl93}, stated in \cite[Lemma 5.4]{CaDiSe3}, implies that $\Phi$ is extendible (and its extendible norm is one).
\end{proof}

The previous lemma is meant to be a tool which will be used in the next proposition and also in the following section. However, as a byproduct of this lemma, a rather unexpected behaviour of diagonal multilinear forms is illustrated. It is easy to see that, if \textit{every} $n$-linear form
on $E_{1} \times \cdots \times E_{n}$ ($E_{j}$ Banach spaces) is extendible, then every $(n-1)$-linear form on any $(n-1)$-tuple of
the previous spaces is extendible. This is a property shared by most classes of multilinear forms, and is related to the notions of ``Property B'', ``coherence'' and ideals ``closed under differentiation'' developed in \cite{BoBrJuPe,BoPe05,CarDimMur09}.
However, this is no longer the case when we restrict ourselves to diagonal multilinear forms. Indeed, Lemma~\ref{ejemplo todas extendibles} shows that every
diagonal trilinear form on $\ell_{1} \times \ell_{1} \times \ell_{2}$ is extendible, while we know from \cite{CaSe2} that there are non-extendible diagonal bilinear forms on $\ell_{1} \times \ell_{2}$.

\medskip
In order to prove the next proposition, recall that, for $s\leq rn$, an $n$-linear form $\phi\in\mathcal L(^n \ep)$ is said to be \textbf{absolutely $(r;s)$-summing} \cite{AM,Ma} if there exists a constant $K>0$ such that for every $x_j^{(k)}\in \ep$, $j=1,\dots , n$, $k=1,\dots , N$,
$$
\left(\sum_{k=1}^N |\phi(x_1^{(k)},\dots x_n^{(k)})|^r\right)^{\frac{1}{r}}\le K\cdot w_s\left((x_1^{(k)})_{k=1}^N\right)\cdots w_s\left((x_n^{(k)})_{k=1}^N\right),
$$
where the weakly $s$-summing norm $w_s$ of the sequence $(x_j^{(k)})_{k=1}^N$ is given by
$$
w_s\left((x_j^{(k)})_{k=1}^N\right)=\sup_{\gamma\in B_{\ell_{p'}}} \left(\sum_{k=1}^N |\gamma(x_j^{(k)})|^s\right)^{\frac1{s}}\, .
$$ There are several possible extensions of the notion of absolutely summing operators to the multilinear setting \cite{CaPe,PG05}. Although the one given above is not considered, in some sense, a good generalization of the linear concept, it will prove useful for our purposes.
Extendible multilinear mappings are those that can be factored through an $\mathcal L_\infty$ space
(due to the injectivity of these spaces). Thus, Grothendieck's multilinear
inequality allows us to derive (see \cite{B88} or \cite[Corollary  2.5]{PG}) that any extendible $n$-linear
form is absolutely $(1;2)$-summing
(for Grothendieck inequality and the notions of cotype and absolutely summing linear operators we refer the
reader to the classical book \cite{DiJaTo95}). From this result and an interpolation technique,
a stronger statement is obtained in \cite[Theorem 3.15]{BMP}: all extendible $n$-linear forms are
absolutely $(r;2r)$-summing, for any $r\ge 1$.

\begin{proposition} \label{pprimasobre2}
Let $1<p<2$ and $n\ge 2$. Then, $\ell_n(\mathcal E,p)= \ell_{\frac{p'}{2}}$.
\end{proposition}

\begin{proof}
Let $\phi_\alpha\in \mathcal E(^n\ell_p)$ be a diagonal extendible $n$-linear form.
Since $\phi_\alpha$ is absolutely $(\frac{p'}{2};p')$-summing, there is a constant $K$ such that,  for all $N$,
$$
\left(\sum_{k=1}^N |\alpha_k|^{\frac{p'}{2}}\right)^{\frac{2}{p'}}=\left(\sum_{k=1}^N |\phi_\alpha( e_k,\dots, e_k)|^{\frac{p'}{2}}\right)^{\frac{2}{p'}}\le K\cdot w_{p'}\left(( e_k)_{k=1}^N\right)^n=K,
$$
which means  that $\alpha$ belongs to $\ell_{\frac{p'}{2}}$.

For the reverse inclusion, let $\alpha\in\ell_{\frac{p'}{2}}$ and define $(\sigma(k))_k=(\alpha(k)^{\frac12})_k$ (with the suitable modification in the real case). We consider the diagonal operator
$D_\sigma:\ell_p\to\ell_1$, given by $D_{\sigma} ((x(k))_{k}) = (\sigma(k) x(k))_{k}$ and we have the following commutative diagram
$$
\xymatrix@C1.5mm{
\ell_p \ar[d]_{D_\sigma} & \times & \ell_p \ar[d]_{D_\sigma} &\times & \ell_p\ar[d]_i &\times& \cdots & \times & \ell_p\ar[d]_i   \ar[rr]^{\phi_\alpha} & & \mathbb K \\
\ell_1&  \times &  \ell_1&\times& \ell_p&\times & \cdots & \times & \ell_p\ar[urr]_{\Phi}
}.
$$
Since $\Phi$ is extendible and has extendible norm one (by Lemma~\ref{ejemplo todas extendibles} and its proof) we obtain that $\phi_\alpha$ is extendible and
$$
\|\phi_\alpha\|_{_{\mathcal E(^n\ell_p)}}\le \|\Phi\|_{_{\mathcal{E}(^n\ell_1\times \ell_1\times\ell_{p}\times\cdots\times \ell_{p})}}\cdot \|D_\sigma\|^2=\|\alpha\|_{\ell_{\frac{p'}{2}}} \, . \qedhere $$
\end{proof}


\section{Diagonal multilinear operators}

In this section, we state and prove our main result. For $p,q,n$ fixed we define the following numbers, that
will be used along the rest of the paper:
\begin{gather}
 r = \Big( \frac{1}{q} - \frac{n}{p} \Big)^{-1} \nonumber \\
t = \max\left\{\left(\frac{n}{p'}+\frac{1}{q}\right)^{-1}, 1 \right\} \label{def:t}.
\end{gather}

\begin{theorem} \label{main} For $n\ge 1$, the following assertions hold.
$\ $
\begin{enumerate}
 \item Let $p=1$,
    \begin{enumerate}
     \item if $1\le q<\infty$, then
	$$
	\ell_n(\mathcal N,1,q)=\ell_n(\mathcal I,1,q)=\ell_q\subsetneq\ell_\infty= \ell_n(\mathcal E,1,q)=\ell_n(\mathcal L,1,q)\, ;
	$$
     \item if $q=\infty$, then
	$$
	\ell_n(\mathcal N,1,\infty)=c_0\subsetneq\ell_\infty= \ell_n(\mathcal I,1,\infty)= \ell_n(\mathcal E,1,\infty)=\ell_n(\mathcal L,1,\infty) \, .
	$$
    \end{enumerate}
  \item Let $1<p<2$,
    \begin{enumerate}
     \item if  $q=1$, then
	$$
	\ell_n(\mathcal N,p,1)=\ell_n(\mathcal I,p,1)=\ell_1\subsetneq \ell_{\frac{p'}{2}}=\ell_n(\mathcal E,p,1)\subsetneq\ell_\infty=\ell_n(\mathcal L,p,1) \, ;
	$$
      \item if $p'<q<\infty$, then
		    $$
	    \ell_n(\mathcal N,p,q)=\ell_n(\mathcal I,p,q)=\ell_t\subsetneq \ell_q=\ell_n(\mathcal E,p,q)\subsetneq\ell_\infty=\ell_n(\mathcal L,p,q)\, ;
	    $$
	  \item if $1<q\le p'$, then, for all $\varepsilon>0$,
		    $$
	    \ell_n(\mathcal N,p,q)=\ell_n(\mathcal I,p,q)=\ell_t\subsetneq \ell_q\subseteq\ell_n(\mathcal E,p,q)\subseteq\ell_{p'+\varepsilon}\subsetneq\ell_\infty=\ell_n(\mathcal L,p,q)\, ;
	    $$
      \item if $q=\infty$, then
	$$
	\ell_n(\mathcal N,p,\infty)=\ell_n(\mathcal I,p,\infty)=\ell_t\subsetneq\ell_\infty= \ell_n(\mathcal E,p,\infty)=\ell_n(\mathcal L,p,\infty) \, .
	$$
    \end{enumerate}
  \item Let  $2\le p<\infty$,
    \begin{enumerate}
     \item if $q=1$, then
	$$
	\ell_n(\mathcal N,p,1)=\ell_n(\mathcal I,p,1)=\ell_n(\mathcal E,p,1)=\ell_1\subsetneq \ell_n(\mathcal L,p,1)=\ell_\infty\text{ or }\ell_r \, ;
	$$
     \item if $1<q<\infty$, then
	$$
	\ell_n(\mathcal N,p,q)=\ell_n(\mathcal I,p,q)=\ell_1\subsetneq \ell_n(\mathcal E,p,q)=\ell_q\subsetneq\ell_n(\mathcal L,p,q)=\ell_\infty\text{ or }\ell_r \, ;
	$$
	      \item if  $q=\infty$, then
	$$
	\ell_n(\mathcal N,p,\infty)=\ell_n(\mathcal I,p,\infty)=\ell_1\subsetneq\ell_\infty= \ell_n(\mathcal E,p,\infty)=\ell_n(\mathcal L,p,\infty) \, .
	$$
    \end{enumerate}
  \item Let $p=\infty$
    \begin{enumerate}
     \item if $q=1$, then
	$$
	\ell_n(\mathcal N,\infty,1)=\ell_n(\mathcal I,\infty,1)=\ell_n(\mathcal E,\infty,1)=\ell_n(\mathcal L,\infty,1)=\ell_1 \, ;
	$$
      \item if  $1< q\le\infty$, then
	$$
	\ell_n(\mathcal N,\infty,q)=\ell_n(\mathcal I,\infty,q)=\ell_1\subsetneq\ell_q= \ell_n(\mathcal E,\infty,q)=\ell_n(\mathcal L,\infty,q) \, .
	$$
    \end{enumerate}
\end{enumerate}
\end{theorem}

The proof of Theorem~\ref{main} will be splitted in several propositions.
The first ones deal with nuclear and integral diagonal mappings.

\begin{proposition} \label{ele1-ele infinito}
Let $T_\alpha\in \mathcal{L}(^{n} \ell_{1},\ell_\infty)$.
Then:
\begin{enumerate}
 \item[(i)] $T_\alpha$ is integral and $\|T_\alpha\|_{_{\mathcal I}}=\|\alpha\|_{_{\ell_\infty}}$.
 \item[(ii)] $T_\alpha$ is nuclear if and only if  $\alpha\in c_0$. In this case, $\|T_\alpha\|_{_{\mathcal N}}=\|\alpha\|_{_{\ell_\infty}}$.
\end{enumerate}
\end{proposition}

\begin{proof} The result follows from the isometric identifications \begin{eqnarray*} \mathcal{L}(^{n} \ell_{1},\ell_\infty)&\cong&\mathcal{L}(^{n+1} \ell_{1}), \\ \mathcal{I}(^{n} \ell_{1},\ell_\infty)&\cong&\mathcal{I}(^{n+1} \ell_{1}), \\ \mathcal{N}(^{n} \ell_{1},\ell_\infty)&\cong&\mathcal{N}(^{n+1} \ell_{1}),\end{eqnarray*} which map diagonal $n$-linear operators  into diagonal $(n+1)$-linear forms, along with the corresponding scalar-valued result
 \cite[Proposition 1.2]{CaDiSe}.
\end{proof}

We need to recall the definition of the injective tensor norm of order $n$. In the $n$-fold tensor product of Banach spaces $E_1\otimes\cdots\otimes E_n$ the
$\varepsilon$-tensor norm is given by
$$
\left\|\sum_{j=1}^N x_1^{(j)}\otimes\cdots\otimes x_n^{(j)}\right\|_{\varepsilon}= \sup_{\gamma_i\in B_{E_i'}} \left|\sum_{j=1}^N \gamma_1(x_1^{(j)})\cdots \gamma_n(x_n^{(j)})\right|.
$$
The space of integral $n$-linear forms on $\ell_{p_{1}} \times\cdots\times \ell_{p_{n}}$ is the dual of the (complete) injective tensor product of the spaces (see the monographs \cite{LibroDeFl93} and \cite{libroRy02} where the 2-fold/bilinear case is treated in detail).

\begin{proposition} \label{nucleares}
Let $p>1$. Then $T_\alpha\in \mathcal{L}(^{n} \ell_{p},\ell_q)$ is nuclear if and only if $\alpha\in\ell_t$, where $t$ is defined in \eqref{def:t}.
In this case, $\|T_\alpha\|_{_{\mathcal{N}}}=\|\alpha\|_{\ell_t}$.
\end{proposition}

\begin{proof}

 If $T_\alpha$ is nuclear, then it  is integral and the
associated $(n+1)$-linear diagonal form $\phi_{_{T_\alpha}}:\ell_p \times\cdots\times\ell_p\times\ell_{q'}\to\mathbb{K}$ is also integral with $||\phi_{_{T_\alpha}}||_{_{\mathcal{I}}}=||T_\alpha||_{_{\mathcal{I}}}$ (for the case $q=1$ we take $\ell_{q'}$ as $c_0$ instead of $\ell_\infty$). Equivalently,  $\phi_{_{T_\alpha}}$ is
$\varepsilon$-continuous. Then, we have
$$
\left|\phi_{_{T_\alpha}}\left(\sum_{k=1}^N
\beta_k  e_k\otimes\cdots\otimes e_k\otimes e'_k\right)\right|\leq
\left\|\phi_{_{T_\alpha}}\right\|_{_{\mathcal{I}}} \left\|\sum_{k=1}^N
\beta_k  e_k\otimes\cdots\otimes e_k\otimes e'_k\right\|_\varepsilon,
$$
which implies that
\begin{eqnarray*}
\left|\sum_{k=1}^N \alpha_k\beta_k\right| & \leq & \left\|T_\alpha\right\|_{_{\mathcal{I}}}\cdot
\left( \sup_{\varphi_1,\dots,\varphi_n\in B_{\ell_{p'}}, \psi\in B_{\ell_q}}\left|\sum_{k=1}^N\beta_k  \varphi_1(e_k)\dots\varphi_n(e_k) . \psi(e_k')\right|\right)\\
                            & = & \left\|T_\alpha\right\|_{_{\mathcal{I}}}\cdot \left(\sup_{\varphi_1,\dots,\varphi_n\in B_{\ell_{p'}}}\left\|\sum_{k=1}^N\beta_k  \varphi_1(k)\dots\varphi_n(k) . e_k'\right\|_{\ell_{q'}}\right)\\
                            & = &                    \left\|T_\alpha\right\|_{_{\mathcal{I}}}\cdot\left\|(T_\beta)_{_N}\right\|_{\B(^n\ell_{p'},\ell_{q'})}\\ & = & \left\|T_\alpha\right\|_{_{\mathcal{I}}}\cdot\left\|(\beta_k)_{k=1}^N\right\|_{\ell_{t'}}.
\end{eqnarray*}
Hence $\alpha\in\ell_{t}$ and $\|\alpha\|_{\ell_{t}}\le \left\|T_\alpha\right\|_{_{\mathcal{I}}}$.

We see now that, if $\alpha\in\ell_t$, then $T_\alpha$ is nuclear and $\|T_\alpha\|_{\mathcal N}\le \|\alpha\|_{\ell_t}$. For $t=1$ the conclusion is immediate, so we assume $t>1$. Let us consider the following factorization:
$$
\begin{array}{rcccrcl}
  \ell_p & \times & \cdots & \times & \ell_p & \overset{T_\alpha}\longrightarrow & \ell_q \\
 D_\eta \downarrow &  &  &  & D_\eta \downarrow & & \uparrow D_\nu \\
  \ell_1 & \times & \cdots & \times & \ell_1 & \underset{\Psi}\longrightarrow &
  \ell_\infty
\end{array}
$$
where $\Psi =T_{(1,1,\dots)}$ and $D_\eta$ and $D_\nu$ are the diagonal linear operators associated with the sequences
$\eta(k)=\alpha(k)^{\frac{t}{p'}}$ and $\nu(k)= \alpha(k)^{\frac{t}{q}}$.

By Proposition \ref{ele1-ele infinito}, the $n$-linear operator $\Psi$ is integral with $\|\Psi\|_{_{\mathcal I}}=1$. Thus, it follows that
$T_\alpha$ is also integral. Since $t>1$, we have $1<p<\infty$ and $\ell_p$ is Asplund. So, $T_\alpha$ is actually nuclear, and its nuclear norm coincides with its integral norm \cite{Al85}. Therefore,
$$
\|T_\alpha\|_{_{\mathcal N}}=\|T_\alpha\|_{_{\mathcal I}}\le \|D_\nu\|\|\Psi\|_{_{\mathcal I}}\|D_\eta\|^n=\|\alpha\|_{_{\ell_t}}.\qedhere
$$
\end{proof}

We finally study the remaining case $p=1$ and $q<\infty$.

\begin{proposition} \label{integrales nucleares}
  Let $q<\infty$ and
$T_\alpha\in \mathcal{L}(^{n} \ell_{1},\ell_q)$. Then the
following are equivalent:
\begin{enumerate}
 \item[(i)] \label{a}  $T_\alpha$ is integral.
 \item[(ii)] \label{b} $T_\alpha$ is nuclear.
 \item[(iii)] \label{c}  $\alpha\in\ell_q$.
\end{enumerate}
When these equivalences hold, $\|T_\alpha\|_{_{\mathcal I}}=\|T_\alpha\|_{_{\mathcal N}}=\|\alpha\|_{_{\ell_q}}$.
\end{proposition}
\begin{proof} The equivalence between \emph{(i)} and \emph{(iii)}
follows as in Proposition~\ref{nucleares}. So we only have to
prove that \emph{(iii)} implies \emph{(ii)}. Given $\alpha\in\ell_q$,
we estimate the nuclear norm of
$\displaystyle T_\alpha^{(s,l)}:=\sum_{k=s}^{s+l} \alpha(k)\cdot e_k'\otimes\cdots\otimes e_k'\cdot
e_k$. For this, we factor
$T_\alpha^{(s,l)}$ as
$$
\begin{array}{rcccrcl}
  \ell_1 & \times & \cdots & \times & \ell_1 & \overset{T_\alpha^{(s,l)}}\longrightarrow & \ell_q \\
 \Pi^{(s,l)}  \downarrow &  &  &  & \Pi^{(s,l)}  \downarrow & & \uparrow D_\alpha^{(s,l)} \\
  \ell_1^{l+1} & \times & \cdots & \times & \ell_1^{l+1} & \underset{\Psi_{l+1}}\longrightarrow &
  \ell_\infty
\end{array}
$$
where $\displaystyle\Pi^{(s,l)}=\sum_{k=s}^{s+l}  e_k'\cdot e_{k-s+1}$ is the
(norm-one) projection on the coordinates $(s,\dots,s+l)$ and
$\displaystyle D_\alpha^{(s,l)}:=\sum_{k=1}^{l+1} \alpha(k+s-1)\cdot e_k'\cdot  e_{k+s-1}$.
From the equalities
$\|{\Psi_{l+1}}\|_{_\mathcal{N}}=\|{\Psi_{l+1}}\|_{_\mathcal{I}}=1$,
and  $\|D_\alpha^{(s,l)}\|=\|
(\alpha_j)_{j=s}^{s+l} \|_{_{\ell_q} }$, we obtain
$$\|T_\alpha^{(s,l)}\|_{_\mathcal{N}}\leq \|D_\alpha^{(s,l)}\|\
\|{\Psi_{l+1}}\|_{_\mathcal{N}} \ \|\Pi^{(s,l)}\|^n\leq \|
(\alpha_j)_{j=s}^{s+l} \|_{_{\ell_q} }.$$ Since $\alpha$ belongs to $\ell_q$, this inequality shows that the series
$\displaystyle \sum_{k=1}^{\infty} \alpha(k)\cdot e_k'\otimes\cdots\otimes e_k'\cdot e_k$ defining
$T_\alpha$ is Cauchy in nuclear norm and thus $T_\alpha$ is nuclear.
\end{proof}

As an immediate consequence of the previous propositions we have
\begin{gather}
\label{elen nucleares}\ell_n(\Nuc,p,q)\overset 1 = \ell_n(\I,p,q) \overset 1 =\ell_t\qquad\textrm{for }(p,q)\not= (1,\infty);\\
\label{elen nucleares2} \nonumber\ell_n(\Nuc,1,\infty)\overset 1 =c_0 \quad\text{ and }\quad\ell_n(\I,1,\infty)\overset 1 =\ell_\infty \, .
\end{gather}

\bigskip

\bigskip

Before turning our attention to extendibility, we comment on another behaviour of diagonal multilinear forms/operators. In general, if two classes of multilinear forms do not coincide on some Banach space $E$, the corresponding classes of vector valued multilinear operators will not coincide (for any range space $F$). More precisely, given ideals of multilinear mappings $\mathfrak{A}$  and $\mathfrak{B}$ and a Banach space $E$, if $\mathfrak{A}(^nE)\ne \mathfrak{B}(^nE)$, then $\mathfrak{A}(^nE,F)\ne \mathfrak{B}(^nE,F)$ for every Banach space $F$. Indeed, if we take a multilinear form $\phi$, say, in  $\mathfrak{A}(^nE)\setminus \mathfrak{B}(^nE)$, then for any nonzero $y\in F$ the multilinear operator $\phi\cdot y$ will belong to $\mathfrak{A}(^nE,F)$ but not to $\mathfrak{B}(^nE,F)$.
It should be noted that, if $\phi$ is a diagonal $n$-linear form (on some sequence space), the operator $\phi\cdot y$ will fail to be diagonal. Let us see that, when restricted to diagonal multilinear forms or operators, things are different.
For example, Proposition~\ref{integrales nucleares} or \eqref{elen nucleares} show that diagonal integral $n$-linear mappings from $\ell_1$ to $\ell_q$ are
nuclear for every $1\le q<\infty$.  Note, however, that there are (scalar-valued) diagonal $n$-linear forms on  $\ell_1$  which are integral but not nuclear, as Theorem~\ref{main forms} shows.

\bigskip

We focus now on the problem of describing the space of diagonal extendible mappings. We take for a moment a more general point of view, considering Banach sequence spaces.
We recall that the K\"othe dual of a Banach sequence space $E$ is defined as
\[
E^{\times} : = \{ z \in \mathbb{K}^{\mathbb{N}}  \colon \sum_{j \in
\mathbb{N}} |z(j) x(j)| < \infty \text{ for all } x \in E \}.
\]
This is a Banach sequence space with the norm given by
$$
\|z \|_{_{E^{\times}}} := \sup_{\| x \|_{_{E}} \leq 1} \sum_{j \in \mathbb{N}} |z(j) x(j)| \, .
$$
(see \cite[page 29]{LiTzaII77} where the analogous notion in the more general context of Banach lattices is developed).
We say that a Banach sequence space $E$ is \textbf{K\"{o}the reflexive} if the canonical inclusion of
$E$ into $E^{\times\times}$ is surjective.

\begin{lemma}\label{extendibles entre bss}
Let $E$ and $F$ be  Banach sequence spaces. Then the diagonal operator $T_\alpha:E\times \cdots\times E\to F$ is extendible for every $\alpha\in F$, and its extendible
norm is at most $\|\alpha\|_{_F}$. In other words,
$$
F\subseteq \ell_n(\E,E,F)
$$
with norm one inclusion.
\end{lemma}

\begin{proof}
We take $\alpha \in F$ and factor $T_\alpha$ as
$$
\xymatrix@C3mm{
E \ar[d]_i &\times & \dots & \times & E\ar[d]_i    \ar[rr]^{T_\alpha} & & F \\
\ell_\infty &\times & \dots & \times & \ell_\infty \ar[urr]_{S_\alpha}
},
$$
where $i$ is the natural inclusion and $S_\alpha$ is just $T_\alpha$ acting on $\ell_\infty\times\cdots\times\ell_\infty$. Since $\ell_\infty$ has the metric extension property,
 $S_\alpha$ is extendible and
  $\|S_\alpha\|_{_{\E(^n\ell_\infty,F)}}=\|S_\alpha\|=\|\alpha\|_{_F}$. Therefore  $T_\alpha\in\E(^nE,F)$ and also $\|T_\alpha\|_{_{\E(^nE,F)}}\leq \|S_\alpha\|\cdot\|i\|^n=\|\alpha\|_{_F}$.
\end{proof}

\begin{remark}\label{identification}\rm A consequence of this lemma was anticipated at the beginning of Section 2: the canonical identification between diagonal $n$-linear operators and $(n+1)$-linear forms does not preserve extendibility. If we take $\alpha$ in $\ell_2\setminus\ell_1$, then the $n$-linear operator $T_\alpha\in\mathcal L(^n\ell_2, \ell_2)$ is extendible by the previous lemma. But Theorem \ref{main forms} tells us that its associated $(n+1)$-linear form is not extendible.
\end{remark}

\begin{proposition} \label{extendible diagonal}
  Let $F$ be a K\"{o}the reflexive sequence space and $2\leq p\leq\infty$. Then $T_\alpha:\ell_p\times \cdots\times \ell_p\to F$ is extendible if and only if $\alpha \in F$. In other words,
$$
\ell_n(\E,p,F) =F\, .
$$
\end{proposition}

\begin{proof}
One direction follows directly from the previous lemma. For the converse, suppose $T_\alpha\in\E(^n\ell_p,F)$. For any $\beta\in F^{\times}$ we denote by $\gamma_\beta: F\to\mathbb K$
the linear functional on $F$ defined by $\beta$:
$$\displaystyle \gamma_\beta(x)=\sum_{k\in\N}\beta(k)\cdot x(k).$$
Then, the scalar valued multilinear form $\phi_{\alpha\beta}=\gamma_\beta\circ T_\alpha$ belongs to $\E(^n\ell_p)$, and has extendible norm not greater than $\|\beta\|_{_{F^\times}}\cdot\|\alpha\|_{_{\ell_n(\E;p,F)}}$.
Since $p\geq 2$, by
\cite[Proposition 3.1]{CaDiSe} we know that $\alpha\beta\in\ell_1$ and $$\|\alpha\beta\|_{_{\ell_1}}\le K_G^{n-1}\cdot\|\phi_{\alpha\beta}\|_{_{\E(^n\ell_p)}},$$
where $K_G$ is the constant in Grothendieck's inequality.
  This shows that $\alpha\in F$ and $\|\alpha\|_{_F}\le K_G^{n-1}\cdot \|\alpha\|_{_{\ell_n(\E,p,F)}}$.
\end{proof}

Note that if we consider an arbitrary Banach sequence space $F$ (not necessarily  K\"{o}the reflexive), the above argument gives the following inclusions:
$$
F\hookrightarrow \ell_n(\E,p,F)\hookrightarrow F^{\times\times}.
$$
Moreover, the result in Proposition 3.1 of \cite{CaDiSe} cited in the proof remains true if we change the space $\ell_p$ (with $p\ge 2$) to any 2-convex Banach sequence space $E$. Hence,
Proposition~\ref{extendible diagonal} is also valid for diagonal multilinear maps from $E\times\cdots\times E$ to a  K\"{o}the reflexive sequence space $F$. For
instance, this applies when $E$ is a Lorentz sequence space $d(w,p)$ with $p\ge 2$.

Taking $F=\ell_q$ in the previous proposition gives
\begin{equation*}
\ell_n(\E,p,q) =\ell_q \, ,
\end{equation*}
for $2\le p \le\infty$ and $1\le q\le \infty$.
If $p=1$, we can also precisely describe the set $\ell_n(\E,p,q)$. This follows from Lemma~\ref{ejemplo todas extendibles}.

\begin{corollary}\label{enele1todosext}
Every diagonal multilinear mapping from $\ell_1\times\cdots\times\ell_1$ to $\ell_q$ ($1\le q\le \infty$) is extendible:
\begin{equation*}
\ell_n(\E,1,q)\overset 1 =\ell_\infty \, .
\end{equation*}
\end{corollary}

\begin{proof}
Given $T_\alpha\in\mathcal L(^n\ell_1,\ell_q)$ we consider its canonical associated diagonal $(n+1)$-linear mapping
$\phi_\alpha:\ell_1\times\cdots\times\ell_1\times \ell_{q'}\to\mathbb K$. By Lemma~\ref{ejemplo todas extendibles}, $\phi_\alpha$ is extendible and its extendible norm equals its usual norm. Then, it is clear  that $T_\alpha$ is extendible with
$$
\|T_\alpha\|_{_{\mathcal E(^n\ell_1,\ell_q)}}\le \|\phi_\alpha\|_{_{\mathcal E(^{n+1}\ell_1\times\cdots\times\ell_1\times \ell_{q'})}}=\|\phi_\alpha\|_{_{\mathcal L(^{n+1}\ell_1\times\cdots\times\ell_1\times \ell_{q'})}}=\|T_\alpha\|_{_{\mathcal L(^n\ell_1,\ell_q)}}. \qedhere
$$
\end{proof}

\bigskip

We have already described $\ell_n(\E,p,q)$ for $p=1$ and $p\ge 2$. For  $1<p<2$, we present a characterization only for the cases $q=1$ and $q>p'$. For the remaining situation ($1<p<2$ and $1<q\le p'$) we just obtain an \emph{estimate} of $\ell_n(\E,p,q)$.

\begin{proposition}\label{ele-ene-exten}
Let $1< p<2$.
\begin{enumerate}
\item For $q=1$, $\ell_n(\mathcal E,p,1)=\ell_{\frac{p'}2}$.

\item For $q>p'$, $\ell_n(\mathcal E,p,q)=\ell_q$.

\item For $1<q\le p'$, $\ell_n(\mathcal E,p,q)\subseteq \ell_{p'+\varepsilon}$, for every $\varepsilon>0$.

\end{enumerate}
\end{proposition}

\begin{proof} \emph{(1)} If $\alpha\in\ell_{\frac{p'}2}$,  the same argument of the proof of Proposition \ref{pprimasobre2} shows that $\phi_\alpha:\ell_p\times\cdots\times\ell_p\times\ell_{\infty}\to\mathbb K$ is extendible with extendible norm not bigger than $\|\alpha\|_{\ell_{\frac{p'}2}}$. Then, $T_\alpha:\ell_p\times\cdots\times\ell_p\to\ell_{1}$ is extendible and $\|T_\alpha\|_{_{\mathcal E(^n\ell_p,\ell_1)}}\le\|\alpha\|_{\ell_{\frac{p'}2}}$.

Suppose now that $T_\alpha\in\mathcal E(^n\ell_p,\ell_1)$. Then, for every $\beta\in\ell_{\infty}$, we have $\phi_\beta\circ T_\alpha=\phi_{\alpha\beta}\in \mathcal E(^n\ell_p)$. By  Proposition \ref{pprimasobre2} we derive that $\alpha\beta=(\alpha(k)\beta(k))_k$ belongs to $\ell_{\frac{p'}{2}}$. Since this happens for every $\beta\in\ell_{\infty}$, we conclude that $\alpha\in\ell_{\frac{p'}{2}}$.

\emph{(2)} Since $q>p'$,  $\ell_q$ has cotype $q>2$. Then, we can apply \cite[Proposition 3.4]{BoPe} to derive that if $T_\alpha\in\mathcal E(^n\ell_p,\ell_q)$ then $T_\alpha$ is absolutely $(q;p')$-summing. So, there exists $K>0$, such that, for every $N$ we have
$$
\left(\sum_{k=1}^N |\alpha_k|^{q}\right)^{\frac{1}{q}}=\left(\sum_{k=1}^N \|\phi_\alpha( e_k,\dots, e_k)\|^{q}\right)^{\frac{1}{q}}\le K\cdot w_{p'}\left(( e_k)_{k=1}^N\right)^n =K.
$$ Hence, $\alpha\in\ell_q$. The other inclusion was already shown in Lemma \ref{extendibles entre bss}.

\emph{(3)} Let $T_\alpha\in\mathcal E(^n\ell_p,\ell_q)$ then, by \cite[Proposition 5.3]{BMP}, $T_\alpha$ is absolutely $(p'+\varepsilon;p')$-summing,  for every $\varepsilon>0$. Reasoning as in the previous item, we see that $\alpha\in\ell_{p'+\varepsilon}$.
\end{proof}


\end{document}